\documentclass[11pt]{amsart}

\usepackage{amsmath,amssymb,amsthm,amsfonts,mathtools}
\usepackage{geometry}
\usepackage{hyperref}
\usepackage{enumitem}
\usepackage{mathrsfs}
\usepackage{tikz-cd}
\usepackage{color}
\usepackage{verbatim}
\usepackage{bm}
\hypersetup{colorlinks=true, linkcolor=blue, citecolor=blue, urlcolor=blue}

\numberwithin{equation}{section}

\newtheorem{theorem}{Theorem}[section]
\newtheorem{proposition}[theorem]{Proposition}

\newtheorem{corollary}[theorem]{Corollary}

\theoremstyle{definition}
\newtheorem{definition}[theorem]{Definition}
\newtheorem{example}[theorem]{Example}

\newtheorem{remark}[theorem]{Remark}

\newcommand{\Ric}{\operatorname{Ric}}
\newcommand{\Hess}{\operatorname{Hess}}

\newcommand{\vol}{\operatorname{vol}}
\newcommand{\MCHess}{\operatorname{MCHess}}
\newcommand{\MHess}{\operatorname{MHess}}

\newcommand{\grad}{\nabla}
\newcommand{\Acal}{\mathcal A}

\DeclareMathOperator{\cs}{cs}
\DeclareMathOperator{\sn}{sn}

\title[Comparison Geometry via Modified Hessians]{Comparison Geometry on Manifolds with Density via Modified Hessians}

\author{Nicholas Ng}

\address{202 Lunt Hall\\
Department of Mathematics\\
Northwestern University\\
Evanston, IL 60208}

\email{nicholas.ng@northwestern.edu}
\urladdr{\url{https://www.nicholasngmath.com}}



\date{}

\begin{document}

\begin{abstract}
Comparison geometry for Bakry--\'Emery Ricci curvature has been extensively developed by Wei--Wylie and others. Motivated by the weighted sectional curvature framework introduced by Wylie and further developed by Kennard--Wylie--Yeroshkin, we study radial comparison geometry on manifolds with density through a modified Hessian arising from this framework.

Under nonnegative weighted sectional curvature together with suitable density control assumptions, we obtain a modified Hessian estimate for the radial function $u = \frac{1}{2}r^2$. From this estimate, we derive Hessian comparison, shape operator comparison, weighted Laplacian comparison, asymptotic radial volume density estimates, and polynomial weighted volume growth bounds. We introduce a normalized weighted radial volume density satisfying a monotonicity property analogous to the radial volume density monotonicity underlying Bishop--Gromov comparison.

We also study rigidity phenomena associated with these comparison estimates. Equality in the Hessian comparison theorem yields radial conformal rigidity, while equality in the modified Hessian estimate forces the metric to have an exact metric cone structure.
\end{abstract}

\maketitle

\section{Introduction}

Comparison geometry studies the geometric consequences of curvature bounds through comparison theorems for distance functions, geodesic spheres, Jacobi fields, and volume growth. Classical sectional curvature bounds lead to Hessian comparison, Laplacian comparison, radial volume density estimates, and rigidity results governing the radial geometry of Riemannian manifolds.

In recent years, several notions of weighted curvature have been introduced in the study of manifolds with density. The Bakry--\'Emery Ricci tensor was introduced by Bakry--\'Emery \cite{BakryEmery}, and comparison geometry for manifolds with density was developed by Wei--Wylie \cite{WeiWylie09} and others. Weighted sectional curvature was introduced by Wylie \cite{Wyl15} and further developed by Kennard--Wylie \cite{KW17} and Kennard--Wylie--Yeroshkin \cite{KWY}. These works established weighted analogues of classical comparison theorems, comparison results for Jacobi fields, convexity results, rigidity phenomena, and sphere theorems. For further comparison results in the weighted setting, see \cite{Wyl15,Wyl16,KW17,KWY}. For background on classical comparison geometry, see \cite{CheegerEbin}.

The present work organizes radial comparison geometry around a modified Hessian tensor. Rather than working through comparison estimates for Jacobi fields in the weighted setting, we derive Hessian comparison and related radial comparison estimates directly from the modified Hessian tensor
$$
\MHess(u)=\Hess u-g(\nabla u,\nabla\varphi)g.
$$
The modified Hessian framework considered here is motivated by the corresponding tensor introduced by Kennard--Wylie--Yeroshkin in \cite{KWY}. The logical structure of the paper is as follows.

$$\begin{array}{c}
\text{Nonnegative weighted sectional curvature} \\
+\ \\ \text{density control}
\end{array}
\Longrightarrow
\, \, \MHess \left(\frac{1}{2}r^2\right) \leq Kg
\, \, \Longrightarrow
\, \, \text{Hessian comparison}$$

$$\Longrightarrow
\, \, \left\{
\begin{array}{c}
\text{shape operator comparison} \\
\text{Laplacian comparison} \\
\text{volume density estimates} \\
\text{asymptotic volume growth estimates} \\
\text{rigidity results}
\end{array}
\right.$$

More precisely, applying the modified Hessian to the radial function $$u =\frac{1}{2}r^2, \qquad r(x)=d(p,x),$$ produces explicit comparison estimates governing the geometry of geodesic spheres. The resulting comparison theory includes the following:

\begin{itemize}
\item Hessian comparison estimates,
\item shape operator comparison estimates,
\item weighted Laplacian comparison estimates,
\item asymptotic radial volume density estimates,
\item polynomial volume growth estimates,
\item rigidity results arising from equality in Hessian comparison,
\item and monotonicity results for normalized weighted radial volume density.
\end{itemize}

In contrast to the classical comparison estimates, the weighted setting introduces the term $\varphi_r$ arising from the density function. In particular, the Hessian comparison theorem can be written as $$
S-\varphi_r I \leq \frac{K}{r}I,$$ where $S$ is the shape operator of the geodesic spheres. The tensor $S-\varphi_r I$ coincides with the tangential part of the weighted shape operator in the radial setting of Kennard--Wylie--Yeroshkin \cite{KWY}.

We also study rigidity phenomena associated with these comparison estimates. Equality in the Hessian comparison theorem yields radial conformal rigidity, while equality in the modified Hessian estimate yields exact conical rigidity.

The paper is organized as follows. Section 2 reviews weighted sectional curvature, modified Hessians, and the relevant radial geometry. Section 3 establishes the modified Hessian estimate. Section 4 derives the Hessian and shape operator comparison theorems. Section 5 establishes Laplacian comparison. Section 6 studies radial volume density estimates and polynomial volume growth. Section 7 proves rigidity in the equality case. Section 8 introduces a normalized weighted radial volume density and proves its monotonicity. Section 9 discusses examples and model geometries.

\section{Preliminaries}

\subsection{Weighted sectional curvature}

Riemannian manifolds $(M,g)$ with smooth density $e^{-\varphi}$ were first studied by Lichnerowicz and later developed by Bakry--\'Emery and others. These manifolds are also referred to as manifolds with smooth measure $\mu$, denoted $(M, g, \mu)$, since choosing a smooth measure $\mu$ is equivalent to choosing a density function. In \cite{Wyl15}, Wylie introduced the notion of \textit{weighted sectional curvature} for a Riemannian manifold $(M,g)$ with density $\varphi$, given by 
$$
\overline{\sec}_\varphi(U,V) = \sec(U,V) + \Hess\varphi(U,U) + d\varphi(U)^2 = g(R^{\nabla^\varphi}(V,U)U,V),
$$
for orthonormal vectors $U, V$, where $\nabla^\varphi=\nabla^{g,\varphi}$ is the \textit{weighted connection} 
$$
\nabla_X^\varphi Y = \nabla_XY - d\varphi(X)Y - d\varphi(Y)X,
$$ 
where $\nabla$ denotes the Levi-Civita connection of the metric $g$; $R^{\nabla^\varphi}$ denotes the corresponding \textit{weighted curvature tensor} 
$$
R^{\nabla^\varphi}(X,Y)Z = \nabla_X^\varphi\nabla_Y^\varphi Z -\nabla_Y^\varphi\nabla_X^\varphi Z - \nabla_{[X,Y]}^\varphi Z.
$$ 
The weighted sectional curvature is closely related to the weighted connection and the $1$-Bakry--\'Emery Ricci tensor 
$$
\Ric_\varphi^1 = \Ric^{\nabla^\varphi}.
$$

Comparison geometry for weighted Ricci curvature has been extensively studied in the work of Wei--Wylie and Wylie--Yeroshkin \cite{WeiWylie09,WY16}. In particular, Wylie--Yeroshkin in \cite{WY16} study comparison geometry for manifolds with density using the general torsion-free affine connection 
$$
\nabla_X^\alpha Y = \nabla_XY - \alpha(X)Y - \alpha(Y)X,
$$ 
where $\alpha$ is a $1$-form. Comparison geometry for weighted sectional curvature was further developed by Kennard--Wylie--Yeroshkin in \cite{KWY}. In particular, they develop weighted Rauch comparison and weighted convexity estimates through the geometry of the weighted connection and the conformal Hessian framework.

The present paper takes a different approach. Rather than emphasizing weighted Rauch comparison and related Jacobi field methods, we study direct radial Hessian estimates associated with the distance function. Under the weighted sectional curvature and density assumptions, we first derive a modified Hessian estimate for the radial function $$u = \frac{1}{2}r^2,$$ from which we derive Hessian comparison, shape operator comparison, Laplacian comparison, radial volume density estimates, volume growth estimates, and rigidity results.

Throughout the paper all geometric quantities are computed with respect to the original metric $g$. Let $(M^n,g)$ be a complete Riemannian manifold and let $\varphi\in C^\infty(M)$. We assume that $M$ has nonnegative weighted sectional curvature 
$$
\overline{\sec}_\varphi \geq 0,
$$ 
together with density control estimates of the form 
$$
a \leq \varphi \leq 0, \qquad |\nabla\varphi| \leq \frac{C}{r}.
$$ 
Under these assumptions, the weighted sectional curvature bound leads naturally to modified Hessian estimates for the radial distance function, which form the starting point for the comparison results developed below.

\subsection{Modified Hessians}

We introduce the modified Hessian tensors that are the central objects of study throughout the paper. Given a manifold with density $(M, g, \varphi)$, consider a conformal metric $\tilde g = e^{-2\varphi}g$. It is well known that for any smooth function $u$, 
\begin{equation}\label{eq:2.1}
\Hess_{\tilde g} u = \Hess_g u + d\varphi \otimes du + du \otimes d\varphi - g(\nabla \varphi, \nabla u)g.
\end{equation}
A basic property in the unweighted setting is that away from the cut locus,
$$
\nabla r \in \ker \Hess_g r
$$
whenever $r$ is a distance function associated with the metric $g$. However, Kennard--Wylie--Yeroshkin in \cite{KWY} observed via (\ref{eq:2.1}) that this is not true in the weighted setting, i.e., $\nabla r$ is not in $\ker \Hess_{\tilde g} r$. To remedy this, Kennard--Wylie--Yeroshkin considered the following lower order perturbation of $\Hess_{\tilde g} r$, \begin{equation}\label{eq:2.2} 
\Hess_{\tilde g}r - d\varphi \otimes dr - dr\otimes d\varphi.\end{equation} 
for which $\nabla r$ becomes an eigenvector ($\nabla u$ lies in the kernel when $u$ is a modified distance function). Kennard--Wylie--Yeroshkin also observed that equation (\ref{eq:2.2}) for a general smooth function $u$ satisfies useful convexity properties along geodesics.

We set notation for the \textit{modified conformal hessian}:
$$
\MCHess u:= \Hess_{\tilde g} u - d\varphi \otimes du - du \otimes d\varphi,
$$ 
where $\tilde g = e^{-2\varphi}g$ and $u$ is any smooth function. Substituting (\ref{eq:2.1}) in the conformal Hessian term of $\MCHess u$ gives the \textit{modified Hessian} 
$$
\MHess u := \Hess_g u - g(\nabla u, \nabla \varphi)g.
$$
The modified conformal Hessian framework considered here is motivated by the modified conformal Hessian tensor introduced by Kennard--Wylie--Yeroshkin in \cite{KWY}. We use the notation $\MCHess u$ when emphasizing the conformal metric formulation, and $\MHess u$ when working directly with the geometry of the metric $g$.

\begin{remark}
All geometric quantities in this paper, including geodesics, distance functions, Hessians, shape operators, and curvature tensors, are computed with respect to the underlying Riemannian metric $g$.
\end{remark}

\subsection{Hessian identities and radial distance functions}\label{subsec:2.3}

In this subsection we derive the basic Hessian identities associated with the radial distance function and the modified Hessian of 
$$
u(r) = \frac{1}{2}r^2, \qquad r(x)=d(p,x)
$$ 
for some fixed point $p \in M$.

Wherever the distance function $r$ is smooth, 
$$
\nabla u=r\nabla r.
$$ 
Moreover, 
$$
\Hess u=dr\otimes dr+r\Hess r.
$$ 
It follows that 
\begin{align*}
    \MHess u &= \Hess u - g(\nabla u, \nabla \varphi)g \\ &= dr \otimes dr + r\Hess r - rg(\nabla \varphi, \nabla r)g \\ &= dr \otimes dr + r\Hess r -r\varphi_rg,
\end{align*} 
where 
$$
\varphi_r := g(\nabla \varphi,\nabla r)
$$
denotes the radial derivative of the density function. In geodesic polar coordinates, if $\varphi = \varphi(r,\theta),$ then $\varphi_r = \partial_r \varphi(r,\theta).$ When restricting to a fixed radial geodesic, we occasionally write $\varphi_r$ as $\varphi'(r)$. Unless otherwise stated, we do not assume that the density function $\varphi$ is globally radial.

\begin{remark}
The function $u = \frac{1}{2} r^2$ is a natural radial function to consider since in Euclidean space, $\Hess\left(\frac{1}{2} r^2\right)=g.$ Thus the failure of this identity measures the deviation of the distance function geometry from the Euclidean model geometry.
\end{remark}

\subsection{Shape operator and Riccati equation}

We briefly review the shape operator and the associated Riccati equation for geodesic spheres. Let
$$
S_r(p)=\{x\in M:d(p,x)=r\}
$$
denote the geodesic sphere of radius $r$ centered at $p$. Wherever $r$ is smooth, the radial direction is given by $\nabla r$, and the Hessian of the distance function determines the second fundamental form of the geodesic spheres.

Define the shape operator by
$$
S(v)=\nabla_v\nabla r.
$$
Then
$$
\Hess r(v,w)=g(Sv,w).
$$

If $\gamma$ is a unit-speed radial geodesic and $J$ is a Jacobi field orthogonal to $\gamma'$, then
$$
S(J)=\nabla_{\gamma'}J.
$$
Using the Jacobi equation, one obtains the matrix Riccati equation for the shape operator.

\begin{proposition}[Riccati Equation]
Along radial geodesics,
$$
S'+S^2+R=0,
$$
where $R(v)=R(v,\gamma')\gamma'$ is the curvature operator along geodesics. 
\end{proposition}

\begin{proof}
Let $J$ be a Jacobi field orthogonal to the radial geodesic $\gamma$. Then 
$$
J''+R(J,\gamma')\gamma'=0.
$$ 
Writing the Jacobi equation in terms of the shape operator of the geodesic spheres yields the Riccati equation.
\end{proof}

\begin{remark}
The operator $S$ describes the second fundamental form of geodesic spheres and satisfies a matrix Riccati equation along radial geodesics. Since
$$
\Hess r(V,W)=g(SV,W),
$$
Hessian comparison theorems may be interpreted as comparison results for solutions of the Riccati equation under curvature bounds.
\end{remark}

See \cite{Lee18,Petersen} for background on Hessian identities, Jacobi fields, and radial comparison geometry.

\subsection{Weighted Laplacian}

The \textit{weighted Laplacian} associated with the density function $\varphi$ is a standard operator in comparison geometry for manifolds with density; see \cite{WeiWylie09}. It is defined by
$$
\Delta_\varphi
=
\Delta-g(\nabla\varphi,\nabla\cdot).
$$
The operator $\Delta_\varphi$ is naturally associated with the weighted measure
$$
d\mu_\varphi=e^{-\varphi}d\vol.
$$

\section{Modified Hessian Estimates}\label{sec:3}

We establish the fundamental modified Hessian estimate underlying the comparison theory developed in subsequent sections.

\begin{theorem}[Modified Hessian Estimate]\label{thm:mchess} 
Let $(M, g, \varphi)$ be a Riemannian manifold with density $\varphi$ such that  
$$
a \leq \varphi \leq 0,
\qquad
|\nabla\varphi| \leq \frac{c}{r}
$$ 
for some constant $c> 0$. Suppose $\overline{\sec}_\varphi \geq 0$ and 
$$
u = \dfrac{1}{2}r^2, \qquad r(x) = d(p, x).
$$ 
Then wherever $r$ is smooth, 
$$
\MCHess u \leq Kg
$$
for some constant $K$.
\end{theorem}

\begin{proof}
Assume $u$ is the modified distance function $u = h \circ r = \frac{1}{2}r^2$ where $h:\mathbb [0, \infty) \rightarrow [0,\infty)$ is defined by $x \mapsto \frac{1}{2}x^2$ so that $h(0) = h'(0) = 0$, and $h'(r) >0$ for $r > 0$.
By \cite[Proposition 3.2]{KWY}, 
$$
\MCHess u = (h'' -  h'\varphi')dr \otimes dr + h'(\Hess_g r - g(\nabla r,  \nabla \varphi)g_r),
$$ 
where $\varphi'=\partial_r\varphi$ denotes differentiation along the chosen radial geodesic; the remaining derivatives are also taken with respect to $r$.

Note 
$$
|\varphi'| = |g(\nabla \varphi, \nabla r)| \leq |\nabla \varphi| \leq \frac{c}{r}.
$$ 
In particular, $ -\varphi' \leq \tfrac{c}{r}$. This together with $u = \tfrac{1}{2}r^2$ yields 
$$
(h'' - h'\varphi') = 1 - r\varphi' \leq 1 + r\left(\dfrac{c}{r}\right) = 1 + c.
$$ 
Thus the radial term 
$$
(h'' - h'\varphi')dr \otimes dr \leq (1 + c)dr \otimes dr.
$$ 
We claim that 
\begin{equation}\label{eq:3.1} 
\Hess_g r - g(\nabla r, \nabla \varphi)g_r \leq  e^{-2\varphi}\dfrac{1}{r}g_r,\end{equation} 
from which it follows that the tangential term 
\begin{align*}
h'(\Hess_g r - g(\nabla r, \nabla \varphi)g_r) &\leq r\left(e^{-2\varphi}\dfrac{1}{r}\right)g_r = e^{-2\varphi}g_r \\ &\leq (1 + c)e^{-2\varphi}g_r\leq (1 + c)e^{-2a}g_r,
\end{align*} 
where we used $c > 0$ and $\varphi \geq a$ in the second to last and last inequality, respectively. Since $-2a \geq -2\varphi \geq 0$ implies $ e^{-2a} \geq e^{-2\varphi} \geq 1$, setting $K := \max\{1 + c, e^{-2a}(1 + c)\} = e^{-2a}(1 + c)$ yields  
$$
\MCHess u \leq K(dr\otimes dr + g_r) =  Kg
$$ 
as desired.

It remains to show (\ref{eq:3.1}). First, recall the re-parametrized distance defined in \cite{KWY} is
$$
s(r) = \int_0^r e^{-2\varphi(t)}dt, \, \, \, \,  \, \, \, \, \, \, \&\, \, \, \, \, \, \,  \, \, \, s(p, q) = \inf \{s \mid \gamma(0) = p, \gamma(1) = q\}.
$$ 
Since 
$$
t \rightarrow \dfrac{\cs_\kappa(t)}{\sn_\kappa(t)}
$$ 
is monotonically decreasing in $t$, we have 
$$
\varphi \leq 0 \implies s(r) = \int_0^r e^{-2\varphi(t)}dt \geq \int_0^r dt = r(q) \implies \dfrac{\cs_\kappa(s(p,  q))}{\sn_\kappa(s(p, q))} \leq \dfrac{\cs_\kappa(r(q))}{\sn_\kappa(r(q))}.
$$ 
With the assumption $\overline{\sec}_\varphi \geq 0$, we have $\overline{\sec}_\varphi(Y, \nabla r) \geq 0$ for orthonormal vectors $Y$ and $\nabla r$ in $T_qM$, where $q$ is some point for which $r$ is smooth. As $Y \perp \nabla r$, we may apply \cite[Theorem 4.16 (1)]{KWY} with $\kappa = 0$ and the preceding inequality to get 
\begin{align*}
(\Hess_g r - g(\nabla r, \nabla \varphi)g_r)(Y, Y) &= (\Hess_g r - g(\nabla r, \nabla \varphi)g)(Y, Y) \\ &\leq e^{-2\varphi(q)}\dfrac{\cs_0(s(p, q))}{\sn_0(s(p, q))} \leq e^{-2\varphi(q)}\dfrac{\cs_0(r(q))}{\sn_0(r(q))} = e^{-2\varphi(q)}\dfrac{1}{r(q)},
\end{align*} 
where $\sn_0(r) = r$ and $\cs_0(r) = \sn_0'(r) = 1$. So 
\begin{equation}\label{eq:3.2}
     (\Hess_g r - g(\nabla r, \nabla \varphi)g_r)(Y, Y) \leq e^{-2\varphi(q)}\dfrac{1}{r(q)} \, \, \, \, \, \forall \, \, Y \in T_qH.
\end{equation} 
Since $\Hess_g r(\nabla r,\cdot)=0,$ the tensor
$$
\Hess_g r-g(\nabla r,\nabla\varphi)g_r
$$
acts only on directions tangent to the geodesic spheres. Thus inequality (\ref{eq:3.2}) yields the tensor inequality (\ref{eq:3.1}).
\end{proof}

\begin{remark}
Since the weighted comparison equations of \cite{KWY} contain explicit radial density terms involving $\varphi_r$, some control on the density is required in order to obtain effective comparison estimates. In a similar manner, the assumptions imposed on $\varphi$ are chosen so that the $\varphi_r$ terms remain controlled in the resulting comparison estimates.
\end{remark}

By similar arguments as in the proof of Theorem \ref{thm:mchess}, we obtain the following conformal modified Hessian estimate.

\begin{proposition}[Conformal Modified Hessian Estimate]\label{prop:3.3} Let $(M, g, \varphi)$ be a Riemannian manifold with density $\varphi$ where $\varphi \leq 0$ and $|\nabla \varphi| \leq \dfrac{c}{r}$ for some constant $c > 0$. Suppose $\overline{\sec}_\varphi \geq 0$ and $u = \dfrac{1}{2}r^2$. Then 
$$
\MCHess u \leq K\tilde g
$$ 
for some constant $K$.
\end{proposition}

\begin{remark}
We need the assumption $\varphi \geq a$ in Theorem \ref{thm:mchess} to get the bound $ e^{-2a}$ for the factor of $e^{-2\varphi}$ that appears after an application of \cite[Theorem 4.16 (1)]{KWY}. For Proposition \ref{prop:3.3}, we do not need this assumption as the factor of $e^{-2\varphi}$ is the weight for the conformal metric $\tilde g$. 
\end{remark}

\section{Hessian Comparison}

The modified Hessian estimate yields radial Hessian and shape operator comparison estimates.

\begin{theorem}[Hessian Comparison]\label{thm:hessian} Let $p \in M$ and $r(x) = d(p, x)$, and assume the hypotheses of Theorem 3.1. Then wherever the distance function $r$ is smooth, 
$$
\Hess r(X,X) \leq \left(\frac{K}{r}+\varphi_r\right)\left|X^\perp\right|^2
$$ 
for every $X \in TM$, where $X = X^\perp + g(X, \nabla r) \nabla r$ with $X^\perp \perp \nabla r$. 
\end{theorem}

\begin{proof}
From Theorem \ref{thm:mchess}, 
$$
dr\otimes dr+r\Hess r-r\varphi_rg\leq Kg
$$ 
for some constant $K$. Solving for $\Hess r$ yields 
$$
\Hess r \leq \left(\dfrac{K}{r} + \varphi_r\right)g - \dfrac{1}{r}(dr \otimes dr).
$$ 
Since $\Hess r(X, X) = \Hess r(X^\perp, X^\perp)$ and $dr(X^\perp) = g(\nabla r, X^\perp) = 0$, the desired estimate follows.
\end{proof}

\begin{corollary}[Shape Operator Comparison]\label{cor:4.2}
The shape operator satisfies
$$
S \leq \left(\frac{K}{r}+\varphi_r\right)I.
$$
\end{corollary}

\begin{remark}
The shape operator comparison can also be viewed from the perspective of the weighted connection of \cite{KWY}. The Hessian comparison estimate can be written in terms of the shape operator $S$ in the equivalent form 
$$
S-\varphi_r I \leq \frac{K}{r}I.
$$ 
Recall that the weighted connection of \cite{KWY} defines the modified shape operator
$$
S^\varphi(X) = \nabla^\varphi_XN = \nabla_XN-d\varphi(X)N-d\varphi(N)X.
$$
If $N=\nabla r$ and $X\perp\nabla r$, then the tangential component satisfies
$$
S^\varphi(X) = S(X)-\varphi_r X.
$$ 
Thus the tensor
$$
S-\varphi_r I
$$
appearing in the present comparison theory agrees with the tangential part of the modified shape operator in the radial setting in \cite{KWY}. The comparison theorem here, however, is formulated directly in terms of the classical radial Hessian and shape operator, with the density term appearing explicitly in the comparison geometry.
\end{remark}

\begin{remark} The shape operator comparison (Corollary \ref{cor:4.2}) gives explicit control on the principal curvatures of geodesic spheres (i.e., the eigenvalues of the shape operator $S$) through the density function $\varphi$.
\end{remark}

\section{Laplacian Comparison}

\begin{corollary}[Laplacian Comparison]\label{cor:5.1}
Assume the hypotheses of Theorem 3.1. Then
$$
\Delta r \leq (n-1)\left(\frac{K}{r}+\varphi_r\right).
$$
\end{corollary}

\begin{proof}
Taking the trace of the Hessian comparison estimate over an orthonormal basis tangent to the geodesic spheres yields the result.
\end{proof}

\begin{remark}
As in the modified Hessian and shape operator comparison estimates, the density $\varphi$ in the weighted setting yields an additional $\varphi_r$ term in the Laplacian comparison estimate. In particular, the density assumptions control the mean curvature of geodesic spheres.
\end{remark}

\begin{corollary}[Weighted Laplacian Comparison] Assume the hypotheses of Theorem 3.1 and let $\Delta_\varphi = \Delta - g(\grad \varphi, \grad \cdot)$ be the weighted Laplacian associated with the density function $\varphi$. Then 
$$
\Delta_\varphi r \leq (n-1)\frac{K}{r}+(n-2)\varphi_r.
$$ 
In particular, if $|\nabla \varphi| \leq \frac{c}{r}$ (or $|\varphi_r| \leq \frac{c}{r}$), then 
$$
\Delta_\varphi r \leq \frac{(n-1)K+(n-2)c}{r}.
$$
\end{corollary}

\begin{proof}
$$
\Delta_\varphi r = \Delta r - g(\nabla \varphi, \nabla r) \leq (n-1)\left(\dfrac{K}{r} + \varphi_r\right) - \varphi_r = (n - 1)\dfrac{K}{r} + (n - 2)\varphi_r.
$$ 
The second statement follows immediately from the hypothesis $\varphi_r \leq \frac{c}{r}$ and the first inequality.
\end{proof}

\section{Radial Volume Density and Volume Growth}
The Laplacian comparison estimate yields radial volume density estimates and volume growth bounds. Let $(M,g)$ be an $n$-dimensional Riemannian manifold, and fix a point $p \in M$. Let $\gamma:[0,\epsilon)\to M$ be a unit-speed geodesic starting at $p$, $\gamma(0)=p$, in a chosen direction $\theta \in S_pM$.

\subsection{Transverse Parallel Frame and Jacobi Fields}

We recall the standard Jacobi field construction in comparison geometry; see \cite{Petersen} for details.

Choose a parallel orthonormal frame along $\gamma$ such that 
$$
\{E_1(t),\dots,E_{n-1}(t)\}
$$
spans $\gamma'(t)^\perp$ for each $t$. For each $i=1,\dots,n-1$, let $J_i(t)$ be the Jacobi field along $\gamma$ satisfying the standard Jacobi equation
$$
J_i''(t) + R(J_i(t), \gamma'(t))\gamma'(t) = 0,
$$
with initial conditions
$$
J_i(0) = 0, \qquad J_i'(0) = E_i(0).
$$

\begin{remark}
The Jacobi fields $J_i(t)$ describe how nearby radial geodesics spread apart in directions orthogonal to $\gamma'(t)$. Equivalently, they measure the radial expansion of small angular sectors centered about the geodesic $\gamma$.
\end{remark}

\subsection{Radial Volume Density}

The radial volume density is determined by the transverse Jacobi fields along radial geodesics. Define the $(n-1)\times(n-1)$ matrix
$$
A(t) = \begin{bmatrix} J_1(t) & \cdots & J_{n-1}(t) \end{bmatrix},
$$
where each Jacobi field $J_i(t)$ is expressed in the basis $\{ E_j(t) \}_{j=1}^{n-1}$.

\begin{definition}[Radial Volume Density]
The \emph{radial volume density} in geodesic polar coordinates centered at $p$ is
$$
\mathcal A(r,\theta) := \det A(r),
$$
where $r$ is the radial distance along $\gamma$, i.e., $r = t$.
\end{definition}

\begin{remark}
Geometrically, $\mathcal A(r,\theta)$ represents the $(n-1)$-dimensional volume of the infinitesimal parallelepiped spanned by geodesics emanating from $p$ in directions orthogonal to $\gamma'$.  
It is the natural density factor appearing in polar coordinates around $p$.
\end{remark}

\subsection{Derivative of the Radial Volume Density and the Laplacian}

We compute the derivative of $\log \mathcal A(r,\theta)$ in the radial direction.

\begin{proposition}[Radial derivative and Laplacian]
Along the geodesic $\gamma$, we have
$$
\partial_r \log \mathcal A(r,\theta) = \Delta r (\gamma(r)),
$$
where $\Delta r$ is the Laplacian of the distance function at $\gamma(r)$.
\end{proposition}

\begin{proof} 
This is a standard consequence of the Jacobi field description of the radial volume density; see \cite{Lee18, Petersen}.
\end{proof}

\begin{proposition}[Asymptotic Radial Volume Density Estimate]\label{prop:6.5}
Assume the hypotheses of the Hessian comparison theorem. Fix a point $p\in M$, and let $\Acal(r,\theta)$
denote the radial volume density in geodesic polar coordinates centered at $p$. Then there exists $r_0>0$ such that for every $r\ge r_0$,
$$
\Acal(r,\theta) \leq C(r_0, \theta)r^{(n-1)K}.
$$
\end{proposition}

\begin{proof}
The Laplacian comparison estimate together with $\partial_r \log \Acal(r, \theta) = \Delta r(\gamma_\theta(r))$ gives 
$$
\partial_r\log \Acal(r, \theta) \leq (n-1)\left(\frac{K}{r}+\varphi'\right).
$$
For $0 < r_0 < r$, integrating the inequality from $r_0$ to $r$ gives 
$$
\log \Acal(r, \theta) - \log \Acal(r_0, \theta) \leq (n - 1)\left[K\ln\left|\frac{r}{r_0}\right| + (\varphi(r) - \varphi(r_0))\right].
$$ 
Exponentiating yields 
$$
\Acal (r, \theta) \leq \Acal(r_0, \theta)\left(\frac{r}{r_0}\right)^{(n -1)K}e^{(n - 1)(\varphi(r) - \varphi(r_0))}.
$$ 
Setting $C(r_0, \theta) = \Acal(r_0, \theta) r_0^{-(n-1)K}e^{-(n - 1)\varphi(r_0)}$ and using $a \leq \varphi \leq 0$ gives the desired estimate. 
\end{proof}

\begin{remark}
Near $r=0$, the radial volume density satisfies the standard Euclidean asymptotics $\Acal(r,\theta)\sim r^{n-1},$ while the radial volume density estimate describes the large-radius behavior of the manifold with density.
\end{remark}

\begin{corollary}[Polynomial Volume Growth]\label{cor:6.7}
Assume the hypotheses of the Hessian comparison theorem. Then there exists a constant $C>0$ such that
$$\vol(B_R(p)) \leq CR^{(n-1)K+1}$$
for sufficiently large $R$.
\end{corollary}

\begin{proof}
By standard normal coordinate expansion, there exists a sufficiently small $r_0 > 0$ such that $\Acal(r, \theta) \sim r^{n -1}$ for all $r \leq r_0$. Then $\Acal(r, \theta) \leq Mr^{n -1}$ for some constant $M$, and 
$$
\int_{S^{n -1}}\int_0^{r_0} \Acal(r, \theta)drd\theta \leq \dfrac{Mr_0^n}{n}\vol(S^{n-1}).
$$ 
Set $C_1 = Mr_0^nn^{-1}\vol(S^{n-1}).$

For $R > r_0$, integrating the estimate in Proposition \ref{prop:6.5} over $[r_0, R]$ gives 
$$
\int_{r_0}^R \Acal(r, \theta)dr \leq C(r_0, \theta)\int_{r_0}^R r^{(n - 1)K}dr = C(r_0, \theta) \left[\frac{R^{(n-1)K + 1} - r_0^{(n-1)K + 1}}{(n - 1)K + 1}\right].
$$ 
So 
$$
\int_{S^{n-1}}\int_{r_0}^R \Acal(r, \theta)drd\theta \leq C_2 R^{(n-1)K + 1} - C_2r_0^{(n -1)K + 1}
$$ 
where 
$$
C_2 = \frac{1}{(n-1)K + 1}\int_{S^{n-1}}C(r_0, \theta)d\theta.
$$ 
Putting all this together,
\begin{align*}
    \vol(B_R(p)) &= \int_{S^{n -1}}\int_0^R \Acal(r, \theta)drd\theta \\  &= \int_{S^{n-1}}\left(\int_0^{r_0}\Acal(r, \theta)dr + \int_{r_0}^R \Acal(r, \theta)dr\right)d\theta \\ &= C_1 + C_2R^{(n-1)K + 1} 
\end{align*} for sufficiently large $R$. Hence 
$$
\vol(B_R(p)) \leq CR^{(n-1)K+1}
$$
for some constant $C>0$.
\end{proof}

\begin{corollary}[Weighted Volume Growth]
Assume the hypotheses of the Hessian comparison theorem. Let $d\mu_\varphi=e^{-\varphi}d\vol$
denote the weighted volume measure. Then there exists a constant $C>0$ such that
$$
\mu_\varphi(B_R(p)) \leq CR^{(n-1)K+1}
$$
for sufficiently large $R$.
\end{corollary}

\begin{proof}
From Corollary \ref{cor:6.7}, there exists a constant $C_1>0$ such that $\vol(B_R(p)) \leq C_1R^{(n-1)K+1}$
for sufficiently large $R$. Since $a\leq \varphi \leq 0,$ we have $1\leq e^{-\varphi}\leq e^{-a}.$
Therefore,
$$
\mu_\varphi(B_R(p)) = \int_{B_R(p)} e^{-\varphi}\, d\vol \leq e^{-a}\vol(B_R(p)) \leq e^{-a}C_1R^{(n-1)K+1}.
$$ 
Setting $C=e^{-a}C_1$ gives the estimate.
\end{proof}

\begin{remark}
The weighted volume growth has the same polynomial order in $R$ as the unweighted volume growth because the density function remains uniformly bounded. If the boundedness assumption on $\varphi$ is relaxed, the density term may contribute nontrivially to the large-radius volume growth.
\end{remark}

\section{Rigidity Results}

We study rigidity associated with equality in the preceding comparison estimates.

\subsection{Umbilicity}

We study the geometric consequences of equality in the Hessian comparison theorem.

\begin{corollary}[Umbilicity]\label{cor:7.1}
Suppose equality holds in Theorem \ref{thm:hessian} for all $v\perp\nabla r.$
Then 
$$
S = \left(\frac{K}{r}+\varphi_r\right)I.
$$ 
Hence geodesic spheres are totally umbilic.
\end{corollary}

\begin{proof}
Suppose equality holds in the Hessian comparison theorem for every vector $X \perp \nabla r$. Then 
$$
\Hess r(X,X) = \left(\frac{K}{r}+\varphi_r\right)|X|^2.
$$ 
Since $\Hess r(X,Y)=g(SX,Y),$ the quadratic form associated with the shape operator agrees with
$$
\left(\frac{K}{r}+\varphi_r\right)g
$$
on the tangent space of the geodesic spheres. Therefore all tangential directions have the same principal curvature, and hence the shape operator is a scalar multiple of the identity, i.e., 
$$
S=\left(\frac{K}{r}+\varphi_r\right)I.
$$
\end{proof}

\subsection{Radial metric evolution}

\begin{proposition}[Radial Metric Evolution]
Suppose equality holds in the Hessian comparison theorem along all directions orthogonal to $\nabla r$. Let $g=dr^2+g_r$ be the metric in geodesic polar coordinates. Then the induced metrics $g_r$ on the geodesic spheres satisfy
$$
\partial_r g_r = 2\left(\frac{K}{r}+\varphi_r\right)g_r.
$$
\end{proposition}

\begin{proof}
The metric evolution equation for geodesic spheres is $\partial_r g_r=2\operatorname{II}_r.$ Since $\operatorname{II}_r(v,w)=g(Sv,w),$ substituting the equivalent expression for $S$ in Corollary \ref{cor:7.1}  yields the desired evolution equation.
\end{proof}

\begin{theorem}[Radial Conformal Rigidity] Assume the hypotheses of Theorem 3.1 and suppose equality holds in the Hessian comparison theorem along all directions orthogonal to $\nabla r$. Wherever the distance function $r$ is smooth, the metric is of the form 
$$
g = dr^2 + r^{2K}e^{2\varphi(r,\theta)}g_0(\theta),
$$
where $g_0$ is the metric on a fixed geodesic sphere.
\end{theorem}

\begin{proof}
Integrating the radial metric evolution equation in the radial direction $\gamma_\theta$ yields
$$
g_r(r,\theta) = \exp\left(2\int \left(\frac{K}{r}+\varphi_r(r,\theta)\right)\,dr\right)g_0(\theta),
$$
where $g_0$ is the metric on a fixed geodesic sphere. Since
$$
\int \frac{K}{r}\,dr = K\log r
\qquad \text{and} \qquad
\int \varphi_r(r,\theta)\,dr = \varphi(r,\theta) + \mathcal F(\theta),
$$
we obtain
$$
g_r(r,\theta) = r^{2K}e^{2(\varphi(r,\theta) + \mathcal F(\theta))}g_0(\theta) = r^{2K}e^{2\varphi(r,\theta)}g_0(\theta),
$$ 
where $e^{2\mathcal F(\theta)}$ is absorbed into the reference metric $g_0(\theta)$. Since the metric in geodesic polar coordinates decomposes as $g = dr^2 + g_r,$ it follows that 
$$
g = dr^2 + r^{2K}e^{2\varphi(r,\theta)}g_0(\theta).
$$
\end{proof}

\begin{remark}
The induced metrics on geodesic spheres evolve by scalar rescaling. In particular, the additional $\varphi_r$ term in the radial metric evolution equation gives rise to the conformal factor $e^{2\varphi(r,\theta)}$ appearing in the rigidity metric.
\end{remark}

\begin{corollary}[Warped Product Rigidity]
If, in addition, the density $\varphi = \varphi(r)$ is radial with respect to $p$, then 
$$
g = dr^2+r^{2K}e^{2\varphi(r)}g_0.
$$
\end{corollary}

\begin{remark}
The rigidity statement is radial and local in nature. It should not be interpreted as a global classification theorem.
\end{remark}

Equality in the modified Hessian estimate rigidifies the radial geometry substantially more strongly than equality in the Hessian comparison theorem.

\begin{theorem}[Conical Rigidity from Equality in the Modified Hessian Estimate]
Assume the hypotheses of Theorem \ref{thm:mchess}. Suppose that $$\MHess u = Kg$$ wherever the distance function $r$ is smooth. Then the metric is locally of the form 
$$
g = dr^2 + r^2g_0,
$$ 
where $g_0$ is the metric on a fixed reference geodesic sphere. Moreover, the density 
$$
\varphi(r, \theta) = (1 - K)\log r + F(\theta)
$$ 
for some function $F$ depending only on angular variables.
\end{theorem}

\begin{proof}
    Recall the identity for $\MHess u$ from Section \ref{subsec:2.3}: 
    $$
    \MHess u = dr \otimes dr + r\Hess r - r\varphi_rg.
    $$ 
    This together with equality in the modified Hessian estimate yields \begin{equation}\label{eq:7.1}dr \otimes dr + r\Hess r - r\varphi_rg = Kg.\end{equation} Evaluating on $\nabla r$, and using that $\Hess r(\nabla r, \nabla r) = 0$ and $dr \otimes dr(\nabla r, \nabla r) = 1,$ we get 
    $$
    1 - r\varphi_r = K,
    $$ 
    or equivalently, 
    $$
    \varphi_r = \dfrac{1 - K}{r}.
    $$ 
    For any $X \perp \nabla r$, evaluating equation \ref{eq:7.1} on $X$ with $\varphi_r$ substituted gives 
    $$
    r\Hess r(X, X) - (1 - K)|X|^2 = K|X|^2.
    $$ 
    Solving for $\Hess r(X, X)$ yields $$\Hess r(X, X) = \frac{1}{r}|X|^2,$$ from which we obtain that the shape operator satisfies 
    $$
    S = \dfrac{1}{r}I.
    $$ 
    Consider the metric $g = dr^2 + g_r$ in geodesic polar coordinates. We integrate the metric evolution equation corresponding to $S$, 
    $$
    \partial_r g_r = \frac{2}{r}g_r,
    $$ 
    to get 
    $$
    g_r = r^2g_0
    $$ for some fixed reference geodesic sphere. Thus 
    $$
    g = dr^2 + r^2g_0.
    $$ 
    Finally, integrating $\varphi_r(r, \theta) = \dfrac{1 - K}{r}$ with respect to $r$ yields the formula for density $\varphi(r, \theta)$.
\end{proof}

\begin{remark}
Equality in the Hessian comparison theorem implies that the geodesic spheres are totally umbilic and evolve conformally along radial directions. However, equality in the modified Hessian estimate is substantially stronger, as it completely determines the radial evolution of the geodesic spheres. In particular, it rigidifies the metric into an exact metric cone structure.
\end{remark}

\begin{remark}
In the equality case of the modified Hessian estimate, the resulting metric is independent of the constant $K$. Moreover, the density function $\varphi$ is constrained to have the form
$$
\varphi(r,\theta) = (1-K)\log r + F(\theta),
$$
where $F$ depends only on the angular variables.
\end{remark}

\section{Normalized Radial Volume Density}

The Hessian comparison theorem suggests a natural normalization for the radial volume density adapted to the weighted setting. In particular, from the proof of Proposition \ref{prop:6.5}, we see that the weighted comparison geometry naturally introduces the radial scaling factor $$r^{(n-1)K}e^{(n-1)\varphi(r, \theta)}.$$ Motivated by this, we define the \textit{normalized radial volume density} as
$$
\widetilde{\Acal}(r,\theta) = \frac{\Acal(r,\theta)}{r^{(n-1)K}e^{(n-1)\varphi(r, \theta)}}.
$$

\begin{proposition}[Monotonicity of the Normalized Density]
The normalized radial volume density satisfies 
$$
\partial_r\log \widetilde{\Acal}(r,\theta) \leq 0.
$$ 
Thus $\widetilde{\Acal}(r,\theta)$
is monotone nonincreasing along radial geodesics.
\end{proposition}

\begin{proof}
Using the definition of $\widetilde{\Acal}$,
$$\log \widetilde{\Acal} = \log \Acal - (n-1)K\log r - (n-1)\varphi(r, \theta).$$
Differentiating with respect to $r$ gives

$$\partial_r \log \widetilde{\Acal} = \partial_r \log \Acal - (n-1)\frac{K}{r} - (n-1)\varphi_r(r, \theta).$$ From substituting the Laplacian comparison estimate, we obtain $$\partial_r \log \widetilde{\Acal} \leq (n - 1) \left(\dfrac{K}{r} + \varphi_r(r, \theta)\right) - (n-1)\frac{K}{r} - (n-1)\varphi_r(r, \theta) = 0.$$
\end{proof}

\begin{remark}
In Euclidean space, $\Acal(r,\theta)=r^{n-1},$ and under the assumption $\Ric \geq 0,$
the normalized radial volume density $\frac{\Acal(r,\theta)}{r^{n-1}}$ is monotone nonincreasing along radial geodesics. This is the standard volume density monotonicity underlying Bishop--Gromov comparison.

In the setting of this paper, nonnegative weighted sectional curvature, $\overline{\sec}_\varphi \geq 0$, together with the weighted Hessian comparison theorem yields the analogue of the Euclidean scaling factor $r^{n-1}$ given by $r^{(n-1)K}e^{(n-1)\varphi(r, \theta)}.$ The preceding proposition states that the normalized radial volume density $\widetilde{\Acal}(r,\theta)$ is monotone nonincreasing along radial geodesics.
\end{remark}

\section{Examples and Model Geometries}

We include here several examples illustrating the comparison estimates and rigidity phenomena developed above.

\begin{example}[Euclidean Case]
If
$$
K=1, \qquad \varphi=0,
$$
then equality in the modified Hessian estimate reduces to the classical identity
$$
\Hess\!\left(\frac12 r^2\right)=g,
$$ 
realized by Euclidean space. In this case, we also recover the standard Hessian and Laplacian comparison estimates.
\end{example}

\begin{example}[Warped Product Equality Case]
Consider the warped product metric 
$$
g=dr^2+r^{2K}g_0.
$$
Recall that for a warped product metric 
$$
g=dr^2+f(r)^2g_0,
$$ 
the shape operator of the geodesic spheres is given by 
$$
S=\frac{f_r}{f}I.
$$ 
Taking $f(r)=r^K,$
we obtain 
$$
S=\frac{K}{r}I.
$$
Thus equality holds in the Hessian comparison theorem when $\varphi=0$. Consequently, equality propagates throughout the corresponding comparison estimates.
\end{example}

\begin{example}[Bounded Radial Density Model]
Let 
$$
\varphi(r)=-\frac{b}{1+r}, \qquad b>0.
$$
Then 
$$
-b \leq \varphi(r) \leq 0, \qquad \text{and} \qquad \varphi_r(r)=\frac{b}{(1+r)^2}.
$$
Thus the terms $\varphi_r \rightarrow 0$ as $r\to\infty$. Consequently, the weighted comparison estimates approach the classical radial comparison estimates at large radius, and the resulting volume growth remains polynomial.
\end{example}

\begin{remark}
These examples illustrate how the additional $\varphi_r$ term arising from the density function affects the radial comparison estimates.
\end{remark}

\bibliographystyle{amsplain}

\end{document}